\documentclass[12pt]{amsart}

\usepackage{a4wide}

\usepackage{amscd}
\usepackage{amssymb}
\usepackage{amsthm}
\usepackage{amsfonts}
\usepackage{amsmath}
\usepackage{latexsym}

\usepackage{epsfig}
\usepackage{graphicx}
\usepackage{epstopdf}

\newtheorem{prop}{Proposition}
\newtheorem{theorem}{Theorem}
\newtheorem{lemma}{Lemma}

\newtheorem{corollary}[theorem]{Corollary}

\newtheorem{example}{Example}

\newcommand{\N}{\ensuremath{\mathbb{N}}}

\newcommand{\R}{\ensuremath{\mathbb{R}}}
\newcommand{\E}{\ensuremath{\mathbb{E}}}
\newcommand{\p}{\ensuremath{\mathbb{P}}}

\newcommand{\C}{\ensuremath{\mathfrak{C}}}
\def\e{{\mathrm{e}}}

\usepackage{color}
\definecolor{darkgreen}{rgb}{0, .5, 0}
\definecolor{darkred}{rgb}{.5, 0, 0}

\allowdisplaybreaks

\title{Robustness of Hilbert space-valued stochastic volatility models}

\author{Fred Espen Benth and Heidar Eyjolfsson}

\date{Draft: \today}

\thanks{We are grateful to Dennis Schroers for valuable discussions and insight. The Thematic Research Group SPATUS, funded by UiO:Energy, is acknowledged.} 

\begin{document}

\begin{abstract}
In this paper we show that Hilbert space-valued stochastic models are robust with respect to perturbation, due to measurement or approximation errors, in the underlying volatility process. Within the class of stochastic volatility modulated  Ornstein-Uhlenbeck processes, we quantify the error induced by the volatility in terms of perturbations in the parameters of the volatility process. We moreover study the robustness of the volatility process itself with respect to finite dimensional approximations of the driving compound Poisson process and semigroup generator respectively, when considering operator-valued Barndorff-Nielsen and Shephard stochastic volatility models. We also give results on square root approximations. In all cases we provide explicit bounds for the induced error in terms of the approximation of the underlying parameter. We discuss some applications to robustness of prices of options on forwards and volatility. 
\end{abstract}

\maketitle

\section{Introduction}
In mathematical finance, the dynamics of forward and futures prices may be modelled as an infinite dimensional stochastic process taking values in some suitable Hilbert space. To capture probabilistic features of the price evolution, operator-valued stochastic volatility models have recently gained some attention in the literature. In this paper, we address the question of robustness of Hilbert-valued stochastic processes with respect to perturbations in the underlying stochastic volatility dynamics.     

The HJM-approach in fixed-income theory (see Heath, Jarrow and Morton \cite{HJM} and Filipovi\'c \cite{Filip}) has been adopted into commodity and energy futures price models (see Benth, \v{S}altyt\.e Benth and Koekebakker \cite{BSBK}). Using the so-called Musiela-parametrization, the forward price $f(t,x)$ at time $t\geq 0$ of a contract delivering delivering a commodity (say oil or coffee beans) in time $x\geq 0$ defines a stochastic process $(f(t,\cdot))_{t\geq 0}$ which takes values in a Hilbert space of real-valued functions on $\mathbb R_+$. By no-arbitrage considerations, it may be modelled as the solution of a stochastic partial differential equation 
\begin{equation}
\label{spde-intro}
df(t,x)=\partial_x f(t,x)dt+\sigma(t)dB(t).
\end{equation}
Here, $\partial_x$ is the derivative-operator with respect to $x$, $B$ is a Wiener process in the Hilbert space and $\sigma(t)$ is some suitable operator-valued stochastic process acting on elements of the Hilbert space. 
We refer to Peszat and Zabczyk \cite{PeZa} for details on such stochastic partial differential equations, and Benth and Kr\"uhner \cite{BeKr1} for a current account on the application to energy and commodity markets. We remark in passing that our analysis is also relevant in more exotic weather derivatives markets as well as fixed-income theory. Moreover, ambit fields as defined in Barndorff-Nielsen, Benth and Veraart \cite{BNBV1} can be re-cast as solutions of equations like \eqref{spde-intro} (see Benth and Eyjolfsson \cite{BeEy17}), which ties our work to a rather general class of random field models where stochastic volatility plays an important role (e.g. turbulence, see \cite{BNBV1}). Ambit fields have been applied to futures price modeling in Di Persio and Perin \cite{DiP-Perin} and Barndorff-Nielsen, Benth and Veraart \cite{BNBV0}.  

In recent years there has been a growing interest in developing and analysing infinite-dimensional stochastic volatility models, that is, stochastic dynamics for $\sigma$ taking values in a suitable space of operators. We refer to Benth, R\"udiger and S\"uss \cite{BeRuSu}, Benth and Simonsen \cite{BeSi}, Cox, Karbach and Khedher \cite{CoxKaKhe1,CoxKaKhe2}, Friesen and Karbach \cite{FriKa}, Cuchiero and Svaluto-Ferro \cite{CuSF}, Benth and Sgarra \cite{BeSg} and Benth and Harang \cite{BeHa} for models of the variance $\mathcal V$, being a dynamics on the positive Hilbert-Schmidt operators, where one defines $\sigma$ as the (positive) square-root of $\mathcal V$. 

The proposed infinite-dimensional volatility models can be divided into two classes. The first extends the Barndorff-Nielsen and Shephard (BNS) stochastic volatility model. The BNS model, suggested by Barndorff-Nielsen and Shephard \cite{BNS}, was later extended to a multivariate context by Barndorff-Nielsen and Stelzer \cite{BNStelzer}. The BNS-model in infinite dimensions belongs to the affine class of models studied in Cox, Karbach and Khedher \cite{CoxKaKhe1,CoxKaKhe2}. The second stream of operator-valued stochastic volatility models is the lifting of the Heston model (see Heston \cite{Heston} for the original one-dimensional model). Benth and Simonsen \cite{BeSi} and Benth, Di Nunno and Simonsen \cite{BeDiNSi} have studied the Heston model infinite dimensions. The present paper takes the infinite-dimensional BNS class of models as its starting point to study robustness.        

When estimating a particular volatility model, or doing numerical simulations of it, one inevitably introduces perturbations of the ``real'' volatility. In this paper we analyse the robustness of the volatility processes itself along with the forward prices, and show that 
the error can be controlled. Indeed, for the volatility, variance and the forward price we are able to show pathwise robustness.

There are some recent accounts on the numerical simulation of such volatility models together with the forward price, i.e., the solution of \eqref{spde-intro}. In Benth {\it et al.} \cite{BeDiNLoPe} a semi-discrete finite difference scheme for the Heston infinite dimensional volatility process is combined with a finite element approximation of \eqref{spde-intro}. Based on regularity results for the model, sharp convergence rates are proven and illustrated in numerical examples. In parallel to our studies, we learned that Karbach \cite{Ka} recently has studied finite rank approximations of affine stochastic volatility models. As mentioned, this class of volatility models includes the infinite dimensional BNS dynamics  that we consider here. Karbach \cite{Ka} shows that there is a sequence of finite rank volatility models which converges weakly in the path space equipped with the Skorohod topology. Moreover, a Galerkin approximation for the corresponding generalized Riccati equations is introduced and studied for the dynamics as in \eqref{spde-intro} with affine volatility. In this paper, we provide explicit bounds in $L^2(\mathbb P)$ of the maximum pathwise error in terms of the various parameters in the volatility model. The pathwise error is measured in the norm of the Hilbert space where the dynamics take values. 

In a large part of this paper, we analyse the difference between the stochastic volatility and its approximation in terms of three main parameters. These are the initial (current) variance, the drift of the BNS volatility model (being a bounded operator on the space of Hilbert-Schmidt operators) and the driving L\'evy process (being a compound Poisson process in the space of Hilbert-Schmidt operators). We establish error bounds on the pathwise supremum error measured by the Hilbert-Schmidt and operator norms, both for the variance model and the square-root of it (the volatility). The estimates for the volatility process rely on results by Birman et al. \cite{BiKoSo75} on fractional powers of selfadjoint operators. Our main result, Theorem \ref{prop:I1_est}, assesses the pathwise error of the Hilbert-valued volatility modulated Ornstein-Uhlenbeck dynamics (as given by \eqref{spde-intro}) in terms of the error induced by the volatility. The result is shown using arguments involving the measurable selection theorem
of Kuratowski and Ryll-Nardzewski \cite{KuRy}.    

Our results are directly applicable to continuity and robustness analysis of options prices considered as functionals of the different parameters in the volatility model. We refer to Benth and Kr\"uhner \cite{BeKr2} for an exposition of various options on infinite dimensional forward price models in commodity markets, including flow forwards. Cuchiero and Svaluto-Ferro \cite{CuSF} analyse options on realised volatility in an infinite dimensional framework. We also want to mention the recent work
by Benth, Schroers and Veraart \cite{BeSchVe1,BeSchVe2} on limit theorems for realized variation of stochastic partial differential equations like \eqref{spde-intro}, which can be utilized for non-parametric estimation of the infinite dimensional stochastic variance.

The paper is structured as follows. In Section~\ref{sec:Forwards} we present the notation we employ in the paper and discuss the stochastic volatility processes as elements in a separable Hilbert space. The main result of this section is that the stochastic volatility-modulated Ornstein-Uhlenbeck process is robust with respect to perturbation of the variance process, meaning that the error such a perturbation causes in the volatility-modulated process can be quantified in terms of the error induced by the variance process. In Section~\ref{sec:Var} we focus our attention on the variance process and present results on how to quantify the error induced by employing approximations. We focus on the case when the variance process is given by an Ornstein-Uhlenbeck type infinite dimensional process which involves an initial value, a bounded semigroup generator, and a compound Poisson process. In particular we elaborate on error bounds in the case of finite dimensional approximations of the compound Poisson process and semigroup generator. In Section~\ref{sec:SqRootPr} we discuss square root approximations of positive definite variance processes and approximative processes. 

\section{Approximation of forwards}\label{sec:Forwards}

Fix a complete filtered probability space $(\Omega,\mathcal F, (\mathcal F_t)_{t\geq 0},\mathbb P)$. Let $H$ denote a separable real Hilbert space, with an inner product $(\cdot,\cdot)_H$, and norm $|\cdot|_H$, and let $\mathcal{H} = L_{HS}(H)$ denote the space of Hilbert-Schmidt operators, i.e. the Hilbert space of bounded operators in $\mathcal T \in L(H) = L(H,H)$ such that $\|\mathcal T \|_\mathcal{H} = \sqrt{\langle \mathcal T, \mathcal T \rangle_\mathcal H} < \infty$, where 
$$
\langle \mathcal S, \mathcal T \rangle_\mathcal H := \sum_{k=1}^\infty (\mathcal S e_k, \mathcal T e_k)_H, 
$$
denotes the inner product of $\mathcal H$ for any $\mathcal S, \mathcal T \in \mathcal H$, and $(e_k)_{k=1}^{\infty}$ is an orthonormal basis (ONB from here on) in $H$. In what follows we shall consider an $H$-valued stochastic volatility modulated Volterra process 
\begin{equation}\label{def:X}
    X(t)=\int_0^t\mathcal S(t-s) \sqrt{\mathcal V(s)}dB(s),
\end{equation}
where $B$ is an $H$-valued Wiener process with covariance operator $Q$ being a positive definite trace class operator on $H$, and $\mathcal V$ is an $\mathcal{H}$-valued non-negative definite stochastic variance process. We assume $\mathcal V$ to be predictable and satisfy the integrability condition
\begin{equation}
\label{int-assumption-V}
    \mathbb E\left[\int_0^t\Vert\mathcal S(t-s)\mathcal V(s)Q^{1/2}\Vert_{\mathcal H}^2ds\right]<\infty
\end{equation}
for all $t\leq T$ for a fixed time horizon $T<\infty$. We refer to Section~\ref{sec:Var} for a specification of the volatility process $\mathcal V$.  Finally, $\mathcal{S}(t)$ is supposed to be a $C_0$-semigroup on $H$. 

Here comes some considerations on the approximation of \eqref{def:X}.
Let $\mathcal V^n$ be some approximation or perturbation of $\mathcal V$, which is supposed to satisfy the same predictability and integrability assumptions \eqref{int-assumption-V} as $\mathcal V$.  We want to assess the error $X(t)-X^{n}(t)$, where
\begin{equation}\label{def:Xn}
    X^{n}(t)=\int_0^t\mathcal S(t-s) \sqrt{\mathcal V^n(s)}dB(s).
\end{equation}
Clearly, it holds for $I^n(t) := X(t)-X^{n}(t)$ that
\begin{equation}\label{def:I1}
    I^n(t) =\int_0^t\mathcal S(t-s)(\sqrt{\mathcal V(s)}-\sqrt{\mathcal V^n(s)})dB(s).
\end{equation}
In this section we want to bound the random variable $\sup_{0 \leq t \leq T} |I^n(t)|_H^2$ in $L^1(\p)$, for the given time horizon $T > 0$. It is relatively straightforward to obtain a $t$-dependent $L^1(\p)$ bound for $|I^n(t)|_H^2$. Our approach is then to demonstrate that the real-valued stochastic process $t \mapsto |I^n(t)|_H^2$ is a.s. continuous and then to apply the measurable selection theorem of Kuratowski and Ryll-Nardzewski (see e.g. subsection 18.3 in Aliprantis and Border \cite{AB} for an excellent exposition). Their selection theorem states that there exists a random variable $\xi$ that takes values in $[0,T]$ such that $\sup_{0 \leq t \leq T} |I^n(t)|_H^2 = |I^n(\xi)|_H^2$. 

To this end, denote by $C([0,T];H)$ the space of all continuous $H$-valued functions on $[0,T]$ equipped with the supremum norm. Let $L^q(0,T;H) := L^q((0,T), \mathcal{B}((0,T)), \text{Leb}; H)$, where $q \geq 1$ and $\text{Leb}$ denotes the Lebesgue measure. The following result is inspired by Chapter 11 in Peszat and Zabczyk \cite{PeZa} and we will use it to prove that the real-valued stochastic process $t \mapsto |I^n(t)|_H^2$ is a.s. continuous. This is needed in the proof of Theorem~\ref{prop:I1_est}. 
\begin{lemma}\label{lem:Xcont}
Suppose that $\mathcal Z(t)$ is a predictable $\mathcal{H}$-valued stochastic process such that 
$$
\int_0^t (t-s)^{-2\alpha} \E\left[ \| \mathcal S(t-s) \mathcal Z(s) Q^{1/2}\|_\mathcal{H}^2  \right] ds  < \infty
$$
holds for some $\alpha \in (0,1/2)$ and all $t \in [0,T]$.
If $\widetilde{X}(t) = \int_0^t \mathcal S(t-s) \mathcal Z(s) dB(s)$, then $\widetilde{X} \in C([0,T];H)$ holds a.s.
\end{lemma}
\begin{proof}
For each $t \in [0,T]$, let 
$$
Y(t) := \frac{1}{\Gamma(1-\alpha)}\int_0^t (t-s)^{-\alpha} \mathcal S(t-s) \mathcal Z(s)dB(s).
$$
Then $Y(t)$ is well defined by the assumptions on the integrand, and moreover
\begin{align*}
    \E[|Y(t)|_H^2] &= \frac{1}{(\Gamma(1-\alpha))^2}\int_0^t (t-s)^{-2\alpha} \E[\|\mathcal{S}(t-s) \mathcal{Z}(s)Q^{1/2} \|_\mathcal{H}^2]ds < \infty
\end{align*}
holds for all $t \in [0,T]$ (see Corollary 8.17 in Peszat and Zabczyk \cite{PeZa}). Let $q > 0$ be a constant such that $1/q < \alpha$. Then, by the Burkholder-Davis-Gundy inequality, see Marinelli and R\"ockner~\cite{MaRo}, 
\begin{align*}
    \int_0^T \E[|Y(t)|_H^q]dt &\leq C\int_0^T \left( \int_0^t  (t-s)^{-2\alpha} \E\left[ \| \mathcal S(t-s) \mathcal Z(s) Q^{1/2}\|_\mathcal{H}^2  \right] ds \right)^{q/2} dt \\ 
    &\leq CT \sup_{t \in [0,T]} \left( \int_0^t (t-s)^{-2\alpha} \E\left[ \| \mathcal S(t-s) \mathcal Z(s) Q^{1/2} \|_\mathcal{H}^2  \right] ds \right)^{q/2} < \infty,
\end{align*}
where $C > 0$ is a constant that depends only on $q$. Hence, it follows that $|Y(\cdot)|_H^q \in L^1(\text{Leb} \times \p)$, and by Fubini's theorem it follows that $|Y(\cdot)|_H^q \in L^1(\text{Leb})$ holds a.s. That is, $Y \in L^q(0,T;H)$ holds a.s. 
Now, for each $\psi \in L^q(0,T;H)$, define the Liouville-Riemann operator $I_\alpha$ by
$$
I_{\alpha} \psi(t) := \frac{1}{\Gamma(\alpha)} \int_0^t (t-s)^{\alpha - 1} \mathcal S(t-s) \psi(s) ds.
$$
By the stochastic Fubini theorem and the semigroup property of $\mathcal S$ it follows that 
\begin{align*}
    \Gamma(\alpha)\Gamma(1-\alpha)I_\alpha Y(t) &= \int_0^t (t-s)^{\alpha - 1} \mathcal S(t-s) \int_0^s (s-r)^{-\alpha} \mathcal S(s-r) \mathcal Z(r)dB(r)ds \\
    &= \int_0^t \mathcal S(t-r) \left( \int_r^t (t-s)^{\alpha - 1}(s-r)^{-\alpha} ds \right) \mathcal{Z}(r)dB(r). 
\end{align*}
Furthermore, by Theorem 11.2 in Peszat and Zabczyk \cite{PeZa} we find that 
\begin{align*}
    \int_r^t (t-s)^{\alpha - 1}(s-r)^{-\alpha} ds &= \int_0^1 (1-z)^{\alpha-1}z^{-\alpha}dz = \frac{\Gamma(\alpha)\Gamma(1-\alpha)}{\Gamma(1)}.
\end{align*}
Therefore $\widetilde{X}(t) = I_\alpha Y(t)$. According to Theorem 11.5 in Peszat and Zabczyk \cite{PeZa}, $I_\alpha$ is a bounded linear operator from $L^q(0,T;H)$ to $C([0,T];H)$, so $\widetilde{X} \in C([0,T];H)$ holds almost surely.
\end{proof}

We are now in a position to prove our main result concerning the robustness of the stochastic volatility process with respect to the variance process.


\begin{theorem}\label{prop:I1_est}
Suppose that 
$$
\E\left[\sup_{0 \leq t \leq T} \|\mathcal{V}(t) - \mathcal{V}^n(t)\|_\mathcal{H}\right] < \infty,
$$
then 
$$
\E\left[\sup_{0 \leq t \leq T}|I^n(t)|_H^2\right] \leq C(T)\E\left[\sup_{0 \leq t \leq T} \|\mathcal{V}(t) - \mathcal{V}^n(t)\|_\mathcal{H}\right],
$$
where we recall the definition of $I^n(t)$ in \eqref{def:I1}. Here, $C(T) = c^2\mathrm{Tr}(Q)(\e^{2kT}-1)(2k)^{-1}$, and $k \in \R$, $c > 0$ are constants such that $\|\mathcal{S}(t)\|_{\rm{op}} \leq c\e^{kt}$.
\end{theorem}
\begin{proof}
By the It\^o isometry we have
\begin{align*}
 \mathbb E[|I^n(t)|_H^2]=\mathbb E[\int_0^t\Vert\mathcal S(t-s)(\sqrt{\mathcal V(s)} - \sqrt{\mathcal V^n(s)} )Q^{1/2}\Vert_{\mathcal H}^2ds]    
\end{align*}
Since the Hilbert-Schmidt norm satisfies $\Vert LK\Vert_{\mathcal H}\leq \Vert L\Vert_{\text{op}}\Vert K\Vert_{\mathcal H}$, for a bounded linear operator $L$ and a Hilbert-Schmidt operator $K$, it follows that
\begin{align*}
\Vert\mathcal S(t-s)(\sqrt{\mathcal V(s)}-\sqrt{\mathcal V^n(s)})Q^{1/2}\Vert_{\mathcal H}
&\leq \Vert\mathcal S(t-s)\Vert_{\text{op}}\Vert(\sqrt{\mathcal V(s)}-\sqrt{\mathcal V^n(s)})Q^{1/2}\Vert_{\mathcal H} \\
&\leq c\e^{k(t-s)}\Vert(\sqrt{\mathcal V(s)}-\sqrt{\mathcal V^n(s)})\Vert_{\text{op}}\Vert Q^{1/2}\Vert_{\mathcal H} \\
&= c\e^{k(t-s)}\mathrm{Tr}(Q)^{1/2}\Vert\sqrt{\mathcal V(s)}-\sqrt{\mathcal V^n(s)}\Vert_{\text{op}}
\end{align*}
In the second inequality, we appealed to the general Hille-Yosida bound of the operator norm of a $C_0$-semigroup. In the last step, we used the well-known identity $\text{Tr}(Q)=\Vert Q^{1/2}\Vert^2_{\mathcal H}$. Next, from Bogachev \cite[Lemma 2.5.1]{Bogachev}, it follows that
$$
\Vert\sqrt{\mathcal V(s)}-\sqrt{\mathcal V^n(s)}\Vert_{\text{op}}^2\leq\Vert\mathcal V(s)-\mathcal V^n(s))\Vert_{\text{op}}.
$$
In addition, we have for a Hilbert-Schmidt operator $L$ that $\Vert L\Vert_{\text{op}}\leq\Vert L\Vert_{\mathcal H}$. Hence, 
\begin{align*}
 \mathbb E[|I^n(t)|_H^2] &\leq c^2\mathrm{Tr}(Q)\mathbb E[\int_0^t \e^{2k(t-s)}\Vert\mathcal V(s)-\mathcal V^n(s)\Vert_{\mathcal H}ds] \\
 &\leq c^2\mathrm{Tr}(Q)\int_0^t \e^{2k(t-s)}ds\E\left[\sup_{0 \leq t \leq T} \|\mathcal{V}(t) - \mathcal{V}^n(t)\|_\mathcal{H} \right] \\
 &\leq C(t) \E\left[\sup_{0 \leq t \leq T} \|\mathcal{V}(t) - \mathcal{V}^n(t)\|_\mathcal{H} \right],
 \end{align*}
 where $C(t)=c^2\mathrm{Tr}(Q)(\e^{2kt}-1)(2k)^{-1}$ for all $t \in [0,T]$.
Now observe that if $0 < \alpha < 1/2$ and $t \leq T$, then by employing analogous arguments as before it holds that 
\begin{align*}
\mathbb E&[\int_0^t (t-s)^{-2\alpha} \Vert \mathcal S(t-s)(\sqrt{\mathcal V(s)} - \sqrt{\mathcal V^n(s)} )Q^{1/2}\Vert_{\mathcal H}^2ds] \\
& \leq c^2\mathrm{Tr}(Q)\E\left[\int_0^t (t-s)^{-2\alpha}\e^{2k(t-s)} \|\mathcal{V}(s) - \mathcal{V}^n(s)\|_\mathcal{H} ds \right] \\
&\leq \int_0^t (t-s)^{-2\alpha}\e^{2k(t-s)}ds \E\left[ \sup_{0 \leq t \leq T} \|\mathcal{V}(t) - \mathcal{V}^n(t)\|_\mathcal{H} \right] < \infty, 
\end{align*}
where the the right-hand side is finite by assumption.
Now we apply Lemma~\ref{lem:Xcont} to conclude that $I^n \in C([0,T];H)$, and therefore $t \mapsto |I^n(t)|_H^2 \in C([0,T];\R)$. According to the measurable selection theorem of Kuratowski and Ryll-Nardzewski~\cite{KuRy} (see also Theorem 8.8 in Peszat and Zabczyk \cite{PeZa}) we conclude that there exists a random variable $\xi$ which takes values in $[0,T]$, and $\sup_{0 \leq t \leq T}|I^n(t)|_H^2 = |I^n(\xi)|_H^2$ holds. Hence, since $\E[C(\xi)] \leq C(T)$, it follows that
\begin{align*}
\E\left[\sup_{0 \leq t \leq T} |I^n(t)|_H^2\right] = \E[\E[ |I^n(\xi)|_H^2 | \xi] ]  \leq C(T)\E\left[\sup_{0 \leq t \leq T} \|\mathcal{V}(t) - \mathcal{V}^n(t)\|_\mathcal{H} \right].
\end{align*}
The proof is completed.



\end{proof}

Consider a financial derivative written on a forward or futures contract on a commodity. Forwards are contracts delivering an underlying, oil, metals, soybeans, power or gas say, at a predetermined (period of) time. Following the analysis in Benth and Kr\"uhner \cite{BeKr2} we can express the price at time $t\geq 0$ of such a  forward as $\mathcal D X(t)$, for some $\mathcal D\in H^*$ and $X(t)$ in \eqref{def:X}. The linear functional $\mathcal D$ models either fixed-delivery contracts, or contracts where delivery takes place over a time period. Assuming that $X$ is modelled directly under the risk-neutral probability measure, the current arbitrage-free price of a European option with exercise time $\tau$ paying $p(\mathcal D X(\tau))$ is given by
\begin{equation}
    P=\mathbb E\left[p(\mathcal D X(\tau))\right].
\end{equation}
Here, we have let the risk-free interest rate be equal to zero for simplicity. If we consider a payoff function $p:\mathbb R\rightarrow\mathbb R$ which is Lipschitz continuous with Lipschitz constant $K > 0$, we find with $X^n$ as in \eqref{def:Xn},
\begin{align*}
    \vert P-P^n\vert&=\left\vert\mathbb E\left[p(\mathcal D X(\tau))\right]-\mathbb E\left[p(\mathcal D X^n(\tau))\right]\right\vert \\
    &\leq K\mathbb E\left[\vert\mathcal D(X(\tau)-X^n(\tau))\vert\right] \\
    &\leq K\Vert\mathcal D\Vert_{\text{op}}\mathbb E\left[\vert X(\tau)-X^n(\tau)\vert_H\right].
\end{align*}
By Theorem \ref{prop:I1_est}, we can bound $\vert P-P^n\vert$ by the norm difference of the variance processes $\mathcal V-\mathcal V^n$. This tells us that the option prices is continuously depending on the variance process. In the next sections we are going to establish precise estimates for the error induced by the difference in variance, giving robustness estimates for the option price.

\section{The variance process}\label{sec:Var}

In this Section we focus on the BNS stochastic volatility model for $\mathcal V$, as introduced in Benth, R\"udiger and S\"uss \cite{BeRuSu} and later studied in an affine extension by Cox, Karbach and Khedher \cite{CoxKaKhe1, CoxKaKhe2}. We introduce various perturbations of $\mathcal V$ and quantify their error to the original volatility process.    

Suppose $\mathcal L(t)$ is an $\mathcal{H}$-valued compound Poisson proess 
\begin{equation}\label{def:CPP}
\mathcal{L}(t) = \sum_{i=1}^{N(t)} \mathcal{X}_i
\end{equation}
where $N(t)$ is a Poisson process with intensity $\lambda > 0$, and $(\mathcal{X}_i)_{i=1}^{\infty}$ is a sequence of independent and identically distributed (i.i.d.) $\mathcal{H}$-valued random variables. We choose the RCLL version of $N(t)$, so that $\mathcal L(t)$ becomes an RCLL L\'evy process with values in $\mathcal H$.  Given the compound Poisson process \eqref{def:CPP}, consider the variance process 
\begin{equation}
\label{def:variance-process-diff}
d\mathcal{V}(t) = \C \mathcal{V}(t)dt + d\mathcal{L}(t),
\end{equation}
where $\mathcal{V}(0) = \mathcal{V}_0$, $\C \in L(\mathcal{H})$ is a bounded operator, and $\mathcal{L}(t)$ is the aforementioned $\mathcal{H}$-valued compound Poisson process \eqref{def:CPP}. It holds that 
\begin{equation}\label{def:variance_process}
\mathcal{V}(t) = \mathfrak{S}(t)\mathcal{V}_0 + \int_0^t \mathfrak{S}(t-s) d\mathcal{L}(s),    
\end{equation}
where $\C$ generates the uniformly continuous $C_0$-semigroup $\mathfrak{S}(t) = \exp(\C t)$ on $\mathcal H$. We remark that since $L(\mathcal{H})$ is a Banach algebra, it follows that $\mathfrak{S}(t) \in L(\mathcal{H})$. Note that the variance process $\mathcal{V}(t)$ depends on the initial position, $\mathcal{V}_0$, the semigroup generator $\C$ and the compound Poisson L\'evy process, $\mathcal{L}(t)$. 
We shall consider approximations of the variance process \eqref{def:variance_process} on the form 
\begin{equation}
\label{def:variance-process-diff-approx}
d\mathcal{V}^n(t) = \C^n \mathcal{V}^n(t)dt + d\mathcal{L}^n(t),     
\end{equation}
where $\mathcal{V}^n(0) = \mathcal{V}^n_0$, and the initial value $\mathcal{V}^n_0$, semigroup generator $\C^n$ and compound Poisson process
\begin{equation}\label{def:CPP_approx}
\mathcal{L}^n(t) = \sum_{i=1}^{N(t)} \mathcal{X}_i^n,
\end{equation}
approximate the initial value, generator and compound Poisson process of another variance process, and the compound Poisson processes $\mathcal{L}(t)$ and $\mathcal{L}^n(t)$ are driven by the same Poisson process, $N(t)$, for each $n \geq 1$. 
In the following subsections we shall study results which quantify the error induced in the variance process by means of approximating the corresponding initial value, compound Poisson process and the semigroup generator, respectively. 

\subsection{Approximation of the compound Poisson process} In this subsection we study finite dimensional approximations of the compound Poisson process  \eqref{def:CPP} by a sequence of compound Poisson processes on the form \eqref{def:CPP_approx}.
\begin{example}
Suppose that $\mathcal{X}_i = (Y_i)^{\otimes 2} = (Y_i,\cdot)_H Y_i$ where $(Y_i)_{i\in\mathbb N}$ are i.i.d. and $H$-valued 
random variables. Given an ONB, $(e_j)_{j\in\mathbb N}$, of $H$ define 
\begin{equation}\label{def:Y_approx}
    Y_i^n = \sum_{j=1}^n (Y_i,e_j)_H e_j,
\end{equation}
for $n \geq 1$. In Proposition~\ref{prop:tensor_approx}, we study the approximation of the i.i.d. $\mathcal{H}$-valued random variables, $\mathcal{X}_i$, i.e. the jumps of the compound Poisson process defined in equation \eqref{def:CPP} with $\mathcal{X}_i^n = (Y_i^{n})^{\otimes 2}$ where $n \geq 1$.
\end{example}
The following result shows that the tensor product forms a locally Lipschitz continuous operator.
\begin{lemma}\label{lem:approx}
If $f,g \in H$, then
$$
\|f^{\otimes 2} - g^{\otimes 2}\|_\mathcal{H}^2 \leq 4(|f|_H^2 \vee |g|_H^2)|f-g|_H^2,
$$
where $x \vee y := \max(x,y)$, for real $x,y$.
\end{lemma}
\begin{proof}
By the definition of the Hilbert-Schmidt norm it holds that,
\begin{align*}
\|f^{\otimes 2} - g^{\otimes 2}\|_\mathcal{H}^2 &= \sum_{j=1}^\infty |( f,e_j )_H f - ( 
g,e_j )_H g |_H^2 \\
&\leq 2\sum_{j=1}^\infty ( f,e_j )_H^2 |f - g |_H^2 + 2\sum_{j=1}^\infty ( f-g,e_j 
)_H^2 |g |_H^2 \\
&= 2(|f|_H^2 + |g|_H^2)|f-g|_H^2,
\end{align*}
where we used the Cauchy-Schwarz inequality together with the triangle inequality, and Parseval's identity in the last equation. The result follows.
\end{proof}
An application of H\"older's inequality now yields the following result.
\begin{prop}\label{prop:tensor_approx}
Suppose that $\mathcal{X}_i = (Y_i)^{\otimes 2}$, where $(Y_i)_{i\in\mathbb N}$ is a sequence of i.i.d. $H$-valued random variables, and $\mathcal{X}_i^n = (Y_i^n)^{\otimes 2}$, where $Y_i^n$ is defined by \eqref{def:Y_approx}, for all $n,i \geq 1$. Then, if $p, q \in [1,\infty]$ verify the equation $1/p + 1/q = 1$ (where $1/\infty = 0$),  it holds that
$$
\E[\|\mathcal{X}_i - \mathcal{X}_i^n \|_\mathcal{H}^2] \leq 4\mathbb E\left[ |Y_i|_H^{2p}\right]^{1/p} \mathbb E\left[ |Y_i-Y_i^n|_H^{2q}\right]^{1/q}.
$$
In particular we have that
\begin{equation}\label{ineq:TensorCase}
    \E[\|\mathcal{X}_i - \mathcal{X}_i^n \|_\mathcal{H}^2] \leq 4\left( \E[|Y_i|_H^4] \E[|Y_i-Y_i^n|_H^4] \right)^{1/2},
\end{equation}
and if $|Y_i|_H \in L^4(\p)$, then the right hand-side of \eqref{ineq:TensorCase} converges to zero as $n \to \infty$.
\end{prop}
\begin{proof}
It follows by Lemma \ref{lem:approx}, and the inequality $|Y_i^n|_H \leq |Y_i|_H$, that 
\begin{equation*}
\|\mathcal{X}_i - \mathcal{X}_i^n\|_\mathcal{H}^2 \leq 4|Y_i|_H^2 |Y_i-Y_i^n|_H^2.
\end{equation*}
The first inequality follows by H\"older's inequality, and the second one by the Cauchy-Schwarz inequality. According to the triangle inequality, 
$$
|Y_i - Y_i^n|_H \leq |Y_i|_H^2 + |Y_i^n|_H^2 \leq 2|Y_i|_H^2.
$$
So, appealing to the assumption, the dominated convergence theorem yields that the right-hand side of \eqref{ineq:TensorCase} converges to $0$ as $n \to \infty$.
\end{proof}
Let us state some useful lemmas: the first lemma is known, but we include it here for the convenience of the reader and for future reference.  
\begin{lemma}\label{lem:U_CPP}
Suppose that $U$ is a separable Hilbert space and $L(t) = \sum_{i=1}^{N(t)} J_i$ is a $U$-valued compound Poisson process, where the Poisson process $N(t)$ has intensity $\lambda > 0$, and $(J_i)_{i\in\mathbb N}$ are i.i.d. $U$-valued random variables. 
\begin{enumerate}
    \item Then it holds that
$$
\E[|L(t)|_U^2] = \lambda t \E[|J_1|_U^2] + \lambda^2 t^2 | \E[J_1]|_U^2,
$$
\item from which it follows that
$$
\E[|L(t)|_U^2] \leq \lambda t(1 + \lambda t) \E[|J_1|_U^2].
$$
\end{enumerate}
\end{lemma}
\begin{proof}
By applying conditioning, we note that
\begin{align*}
\E\left[| L(t) |_U^2 \right] &= 
 \E[ \E[ | \sum_{i=1}^{N(t)} J_i |_U^2 | N(t)]] 
= \sum_{k=0}^\infty \e^{-\lambda t} \frac{(\lambda t)^k}{k!} \E[ | 
\sum_{i=1}^{k} J_i|_U^2 ].
\end{align*}
Since $\langle J_i, J_i \rangle_U = |J_i|_U^2$, $\E[\langle J_i, J_j \rangle_U] = \langle \E[J_i], \E[J_j] \rangle_U$, when $i \ne j$  and the $(J_i)_{i\in\mathbb N}$ are i.i.d., it follows that 
\begin{align*}
    \E[ | \sum_{i=1}^{k} J_i|_U^2 ] &= \sum_{i,j=1}^k \E[\langle J_i, J_j \rangle_U] = k\E[|J_1|_U^2] + k(k-1)|\E[J_1]|_U^2.
\end{align*}
Putting these equations together, it follows that 
$$
\E[|L(t)|_U^2] = \lambda t \E[|J_1|_U^2] + \lambda^2 t^2 | \E[J_1]|_U^2.
$$
Finally, the inequality follows from 
$$
|\E[X_1]|_U^2 \leq (\E[|J_1|_U])^2 \leq \E[|J_1|_U^2],
$$
where we have used Jensen's inequality twice (or the Cauchy-Schwarz inequality for the second bound).
\end{proof}
 
The following lemma will be employed for the Hilbert space $\mathcal{H}$ and the Banach space of so-called trace class operators, i.e. compact operators with absolutely convergent series of eigenvalues.
\begin{lemma} \label{lem:Var_CPP}
Suppose that $\Vert \cdot \Vert_{\mathcal B}$ denotes the norm of a Banach space $\mathcal B$, which is a subspace of $L(H)$ and assume that $\C \in L(\mathcal{B})$. Suppose moreover that $\mathcal{V}(t)$ and $\mathcal{V}^n(t)$ are the variance processes defined by \eqref{def:variance-process-diff}  and \eqref{def:variance-process-diff-approx} respectively, where $\C^n = \C$ and the compound Poisson processes $\mathcal{L}(t)$ and $\mathcal{L}^n(t)$ are driven by the same Poisson process, $N(t)$. Then it holds that 
$$
\Vert\mathcal V(t)-\mathcal V^n(t)\Vert_{\mathcal B} \leq e^{\Vert\mathfrak C\Vert_{\text{op}} t}\left( \Vert \mathcal{V}_0 - \mathcal{V}_0^n \Vert_\mathcal{B} + \sum_{i=1}^{N(t)}\Vert\mathcal X_i-\mathcal X_i^n\Vert_{\mathcal B}\right).
$$
\end{lemma}
\begin{proof}
Let $(T_i)_{i=1}^{\infty}$ be the (exponentially distributed) jump times of the Poisson process $N(t)$, which the compound Poisson processes $\mathcal{L}(t)$ and $\mathcal{L}^n(t)$ have in common. Then, by the triangle inequality and the assumption $\C \in L(\mathcal{B})$, 
\begin{align*}
\Vert\mathcal V(t)-\mathcal V^n(t)\Vert_{\mathcal B} &=\Vert \mathfrak S(t)(\mathcal V_0 - \mathcal V_0^n) + \sum_{T_i\leq t}\mathfrak S(t-T_i)(\mathcal X_i-\mathcal X_i^n)\Vert_{\mathcal B}\\
&\leq \Vert \mathfrak S(t) \Vert_{\text{op}}\Vert \mathcal{V}_0 - \mathcal{V}_0^n \Vert_\mathcal{B} + \sum_{T_i\leq t}\Vert\mathfrak S(t-T_i)\Vert_{\text{op}} \Vert\mathcal X_i-\mathcal X_i^n\Vert_{\mathcal B}\\
&\leq \e^{\Vert \mathfrak C\Vert_{\text{op}}t}\Vert \mathcal{V}_0 - \mathcal{V}_0^n \Vert_\mathcal{B} + \sum_{T_i\leq t}e^{\Vert\mathfrak C\Vert_{\text{op}} (t-T_i)}\Vert\mathcal X_i-\mathcal X_i^n\Vert_{\mathcal B}\\
&\leq e^{\Vert\mathfrak C\Vert_{\text{op}} t}\left( \Vert \mathcal{V}_0 - \mathcal{V}_0^n \Vert_\mathcal{B} + \sum_{i=1}^{N(t)}\Vert\mathcal X_i-\mathcal X_i^n\Vert_{\mathcal B}\right),
\end{align*}
where in the second to the last inequality we used the exponential bound of the semigroup.
\end{proof}

Now we investigate the error induced on the compound Poisson and variance processes by means of approximating $\mathcal{X}_i$ with $\mathcal{X}_i^n$, where we do not assume any specific form for the jump-size approximation, $\mathcal{X}_i^n$.
\begin{prop}\label{prop:Var_conv1}
Suppose for each $n \geq 1$ that $(\mathcal{X}_i)_{i\in\mathbb N}$ and $(\mathcal{X}_i^n)_{i\in\mathbb N}$ are two sequences of $\mathcal{H}$-valued i.i.d. random variables, and moreover that $\| \mathcal{X}_i\|_\mathcal{H}, \| \mathcal{X}_i^n \|_\mathcal{H} \in L^2(\p)$, for all $n,i \geq 1$. Then $\mathcal{L}(t)$ and $\mathcal{L}^n(t)$, as defined by \eqref{def:CPP} and \eqref{def:CPP_approx}, respectively, are square-integrable compound Poisson processes. Suppose furthermore that $\mathcal{V}(t)$ and $\mathcal{V}^n(t)$ are the variance processes defined by \eqref{def:variance-process-diff} and \eqref{def:variance-process-diff-approx} respectively, where $\C^n = \C \in L(\mathcal{H})$ for all $n \geq 1$ and the compound Poisson processes $\mathcal{L}(t)$ and $\mathcal{L}^n(t)$ are driven by the same Poisson process $N(t)$. Then for every $T>0$,
\begin{align*}
\E\left[\sup_{0\leq t\leq T}\|\mathcal{V}(t) - \mathcal{V}^n(t)\|_\mathcal{H}^2\right] &\leq C_0(T)\E[\| \mathcal{V}_0 - \mathcal{V}_0^n \|_\mathcal{H}^2] + C_1(T)\E[\| \mathcal{X}_1 - \mathcal{X}_1^n \|_\mathcal{H}^2]
\end{align*}
where $C_0(T) = 2e^{2T\Vert\mathfrak C\Vert_{\text{op}}}$ and $C_1(T)=2e^{2T\Vert\mathfrak C\Vert_{\text{op}}} T \lambda (1+\lambda T)e^{2T\Vert\mathfrak C\Vert_{\text{op}}}$.
\end{prop}
\begin{proof}
By Lemma \ref{lem:U_CPP} (part 2) it holds that
\begin{align*}
    \E[\|\mathcal{L}(t)\|_\mathcal{H}^2] \leq \lambda t(1 + \lambda t)\E[\|\mathcal{X}_1 \|_\mathcal{H}^2]  < \infty,
\end{align*}
so $\mathcal{L}(t)$ and $\mathcal{L}^n(t)$ (by applying the same reasoning for $\mathcal L^n$) are square integrable. According to Lemma \ref{lem:Var_CPP} (using that $\mathcal H$ is Hilbert with $\C \in L(\mathcal{H})$) it holds that
\begin{align*}
\Vert\mathcal V(t)-\mathcal V^n(t)\Vert_{\mathcal H}
&\leq e^{\Vert\mathfrak C\Vert_{\text{op}} t}\left( \Vert \mathcal{V}_0 - \mathcal{V}_0^n \Vert_\mathcal{H} + \sum_{i=1}^{N(t)}\Vert\mathcal X_i-\mathcal X_i^n\Vert_{\mathcal H}\right).
\end{align*}
It moreover follows by the triangle inequality and square-integrability assumption that 
\begin{align*}
\E[\|\mathcal{X}_i - \mathcal{X}_i^n\|_\mathcal{H}^2] \leq 2(\E[\|\mathcal{X}_i\|_\mathcal{H}^2] + \E[\|\mathcal{X}_i^n\|_\mathcal{H}^2]) < \infty.   
\end{align*}
Hence, by Lemma \ref{lem:U_CPP} we can conclude that $t\mapsto \sum_{i=1}^{N(t)}\Vert\mathcal X_i-\mathcal X_i^n\Vert_{\mathcal H}$ is a square-integrable compound Poisson process with values in $\mathbb R_+$ (a subordinator).
Thus, by appealing to the inequality $(x+y)^2 \leq 2(x^2+y^2)$ and Lemma \ref{lem:U_CPP} 
\begin{align*}
\E\left[\sup_{0\leq t\leq T}\|\mathcal{V}(t) - \mathcal{V}^n(t)\|_\mathcal{H}^2\right] &\leq 2e^{2T\Vert\mathfrak C\Vert_{\text{op}}}
\E\left[\Vert \mathcal{V}_0 - \mathcal{V}_0^n \Vert_\mathcal{H}^2 + \sup_{0\leq t\leq T}\left(\sum_{i=1}^{N(t)}\Vert\mathcal X_i-\mathcal X_i^n\Vert_{\mathcal H}\right)^2\right] \\
&\leq 2e^{2T\Vert\mathfrak C\Vert_{\text{op}}}\left( \E[\Vert \mathcal{V}_0 - \mathcal{V}_0^n \Vert_\mathcal{H}^2] +
\E\left[\left(\sum_{i=1}^{N(T)}\Vert\mathcal X_i-\mathcal X_i^n\Vert_{\mathcal H}\right)^2\right]\right) \\
&\leq 2e^{2T\Vert\mathfrak C\Vert_{\text{op}}} \left( \E[\Vert \mathcal{V}_0 - \mathcal{V}_0^n \Vert_\mathcal{H}^2] + \lambda T (1 + \lambda T)\E\left[\Vert\mathcal X_1-\mathcal X_1^n\Vert_{\mathcal H}^2\right]\right).
\end{align*}
The result follows.
\end{proof}
We can easily obtain the following corresponding result for the approximation of the compound Poisson process.
\begin{corollary}
Suppose for each $n \geq 1$ that $(\mathcal{X}_i)_{i\in\mathbb N}$ and $(\mathcal{X}_i^n)_{i\in\mathbb N}$ are two sequences of $\mathcal{H}$-valued i.i.d. random variables, and moreover that $\| \mathcal{X}_i\|_\mathcal{H}, \| \mathcal{X}_i^n \|_\mathcal{H} \in L^2(\p)$, for all $n,i \geq 1$. 
 If the compound Poisson processes $\mathcal{L}(t)$ and $\mathcal{L}^n(t)$, as defined by \eqref{def:CPP} and \eqref{def:CPP_approx}, respectively, are driven by the same Poisson process $N(t)$, then for every $T>0$,
\begin{align*}
\E\left[\sup_{0\leq t\leq T}\|\mathcal{L}(t) - \mathcal{L}^n(t)\|_\mathcal{H}^2\right] &\leq C(T)\E[\| \mathcal{X}_1 - \mathcal{X}_1^n \|_\mathcal{H}^2]
\end{align*}
where $C(T)=2T \lambda (1+\lambda T)$. 
\end{corollary}
\begin{proof}
The result follows from Proposition \ref{prop:Var_conv1} by setting $\mathcal{V}_0^n = \mathcal{V}_0$ for all $n$ and $\C=0$.
\end{proof}

Given an ONB $(e_j)_{j\in\mathbb N}$ of $H$, introduce by tensorization elements in $L(H)$ defined by
\begin{equation}\label{def:HS_basis}
e_j \otimes e_k = (e_k,\cdot)_He_j,
\end{equation}
for each $(j,k) \in \N^2$. It is well-known (see e.g. Proposition 3.4.14 in Pedersen~\cite{Pedersen}) that the family $(e_j\otimes e_k)_{(j,k)\in\mathbb N^2}$ defines an ONB of $\mathcal H$.  

Recall that an operator $A \in L(U)$, where $U$ is a separable Hilbert space, is said to be \emph{compact} if it transforms any bounded subset of $U$ into a set whose closure is compact (i.e. being relatively compact). Note that if $ \mathcal{T} \in \mathcal{H}$, then $\mathcal{T}$ is compact, see Proposition 3.4.8 in Pedersen~\cite{Pedersen}. Recall furthermore that an operator $A \in L(U)$ is said to be \emph{diagonalizable} if there exists an ONB of $U$ which consists of eigenvectors of $A$. If an operator $\mathcal{T}$ is compact and \emph{normal}, meaning that $\mathcal{T}\mathcal{T}^* = \mathcal{T}^*\mathcal{T}$ holds where $\mathcal{T}^*$ is the adjoint of $\mathcal{T}$, then it is diagonalizable, i.e. there exists an ONB of eigenvectors, $(e_k)_{k\in\mathbb N}$ of $\mathcal{T}$, and the corresponding eigenvalues vanish at infinity (see Theorem 3.3.8 in Pedersen~\cite{Pedersen}). Moreover, from above we know that $(e_j \otimes e_k)_{(j,k)\in\mathbb N^2}$ then becomes an ONB of $\mathcal{H}$. In what follows we shall use the notation $(E_j)_{j\in\mathbb N}$ for an ONB of $\mathcal{H}$. Note in particular that if $d: \N \to \N^2$, is a bijective correspondence, then, we can let $E_j = \psi_{d(j)}$, where $\psi_{(j,k)} = e_j \otimes e_k$ is defined by \eqref{def:HS_basis} and $(e_j)_{j\in\mathbb N}$ is an ONB on $H$.

Suppose that $(\mathcal{U}_n)_{n\in\mathbb N}$ is a nested sequence of finite-dimensional subspaces of $\mathcal{H}$, meaning that $\mathcal{U}_n \subset \mathcal{U}_{n+1}$, for $n \geq 1$. Denote by $\Pi_n \in L(\mathcal{H})$ the orthogonal projection onto $\mathcal{U}_n$ and let  
\begin{equation}\label{def:Tn}
    \mathcal{T}^n := \Pi_n\mathcal{T}.
\end{equation}
Note in particular that 
$\|\mathcal{T}^n\|_\mathcal{H} \leq \|\mathcal{T}\|_\mathcal{H}$ holds for all $n \geq 1$ since $\Vert\Pi_n\Vert_{\text{op}}\leq 1$. So, $\mathcal{T}^n \in \mathcal{H}$.
\begin{prop}\label{prop:HS_approx}
Suppose that $\mathcal{T} \in \mathcal{H}$ and $\mathcal{T}^n$ is defined in equation \eqref{def:Tn} where $\mathcal{U}_n = \mathop{\rm span}\{e_j \otimes e_k : j+k \leq n\}$ and $(e_j)_{j\in\mathbb N}$ is the ONB of $H$. Then it holds that
$$
\|\mathcal{T} - \mathcal{T}^n\|_\mathcal{H}^2 = \sum_{j+k > n} (\mathcal{T}e_k,e_j)_H^2.
$$
If furthermore the basis $(e_j)_{j\in\mathbb N}$ consists of the  eigenfunctions of $\mathcal{T}$ and $(\lambda_j)_{j\in\mathbb N}$ are the corresponding eigenvalues, then $(\lambda_j)_{j\in\mathbb N} \in \ell^2$ and 
$$
\|\mathcal{T} - \mathcal{T}^n\|_\mathcal{H}^2 = \sum_{k > n/2} \lambda_k^2,
$$
\end{prop}
\begin{proof}
Note that, according to Parseval's identity on $\mathcal{H}$ it holds that
\begin{align*}
  \|\mathcal{T} - \mathcal{T}^n\|_\mathcal{H}^2 &=  \|\sum_{j+k > n} \langle \mathcal{T}, e_j \otimes e_k \rangle_\mathcal{H} e_j \otimes e_k \|_\mathcal{H}^2 \\
  &= \sum_{j+k > n} \langle \mathcal{T}, e_j \otimes e_k \rangle_\mathcal{H}^2 \\
  &= \sum_{j+k > n} (\mathcal{T}e_k,e_j)_H^2. 
\end{align*}
Here, we applied the identity 
$$
\langle \mathcal{T}, e_j \otimes e_k \rangle_\mathcal{H} = \sum_{m=1}^\infty (\mathcal{T}e_m, (e_k,e_m)_He_j)_H = (\mathcal{T}e_k,e_j)_H
$$
in the last step. Suppose now that $(e_j)_{j\in\mathbb N}$ is the basis of eigenvalues of $\mathcal{T}$. Then 
$$
\sum_{j=1}^\infty \lambda_j^2 = \sum_{j=1}^\infty |\mathcal{T}e_j|_H^2 = \|\mathcal{T}\|_\mathcal{H}^2 < \infty,
$$
and 
$$
\sum_{j+k > n} (\mathcal{T} e_k,e_j)_H^2 = \sum_{j+k > n} \lambda_k^2(e_k,e_j)_H^2 = \sum_{k > n/2} \lambda_k^2
$$
and the proof is complete.
\end{proof}
We remark that if we let $\mathcal{T} = \mathcal{X}$ in Proposition \ref{prop:Var_conv1}, the eigenvalues are random variables. Furthermore, according to Proposition~ \ref{prop:HS_approx}, if the operator $\mathcal{T} \in \mathcal{H}$ is diagonalizable with an ONB of $\mathcal{H}$ which consists of the eigenfunctions of $\mathcal{T}$, then the corresponding sequence of eigenvalues vanish at infinity since it is in the sequence space $\ell^2$. Hence, according to Theorem 3.3.8. in Pedersen~\cite{Pedersen} it follows that $\mathcal{T}$ is normal and compact.

\subsection{Finite dimensional approximation of the semigroup generator}

Now, before proceeding with the general results, let us consider specific examples of operators $\C \in L(\mathcal{H})$ and their finite dimensional approximations $\C^n$. Motivated from Benth, R\"udiger and S\"uss \cite{BeRuSu}, let
\begin{equation}\label{Def:C1}
\C_1 \mathcal{T} = \mathcal{C}\mathcal{T}\mathcal{C}^*,
\end{equation}
and
\begin{equation}\label{Def:C2}
\C_2 \mathcal{T} = \mathcal{C}\mathcal{T} + \mathcal{T}\mathcal{C}^*,
\end{equation}
where $\mathcal{C} \in L(H)$ and $\mathcal{C}^*$ is the adjoint of $\mathcal{C}$. We notice that the operator $\mathfrak{C}_2$ is an infinite-dimensional extension of the matrix-valued volatility model considered in Barndorff-Nielsen and Stelzer \cite{BNStelzer}. We start by exploring the case of finite dimensional approximations of compact operators. 
\begin{prop}
Suppose $\C_i$, $i=1,2$ is defined by \eqref{Def:C1} and \eqref{Def:C2} respectively, where $\mathcal{C}$ is a normal and compact operator. Then there exists an ONB $(e_j)_{j\in\mathbb N}$ of $H$ consisting of eigenfunctions of $\mathcal{C}$ and $(\lambda_j)_{j \in \mathbb N}$ is the corresponding sequence of eigenvalues which vanishes at infinity. It furthermore holds that $(e_j \otimes e_k)_{(j,k)\in\mathbb N^2}$ is an ONB of $\mathcal{H}$ consisting of eigenfunctions of $\C_i$, i.e. $\C_i(e_j \otimes e_k) = \Lambda_{j,k}(e_j \otimes e_k)$ holds for all $j,k \geq 1$, where the eigenvalues are given by $\Lambda_{j,k} = \lambda_j\lambda_k$ when $i=1$ and $\Lambda_{j,k} = \lambda_j + \lambda_k$ when $i=2$. It moreover holds that $\C_1$ is a compact operator, and that $\C_2$ is not a compact operator.
\end{prop}
\begin{proof}
The existence of the ONB with eigenvalues that vanish at infinity follows by Theorem 3.3.8 in Pedersen~\cite{Pedersen}. We recall that according to Proposition 3.4.14 in Pedersen~\cite{Pedersen} $(e_j \otimes e_k)_{(j,k)\in\mathbb N^2}$ is a basis of $\mathcal{H}$. If $f \in H$, then
\begin{align*}
\C_1(e_j \otimes e_k)f &= (e_k,\mathcal{C}^*f)_H\mathcal{C}e_j =  (\mathcal{C}e_k,f)_H\mathcal{C}e_j \\
&= \lambda_j\lambda_k (e_k,f)_He_j = \lambda_j\lambda_k(e_j \otimes e_k)f,
\end{align*}
and
\begin{align*}
\C_2(e_j \otimes e_k)f &= (e_k,f)_H\mathcal{C}e_j + (e_k,\mathcal{C}^*f)_He_j \\
&= (e_k,f)_H\mathcal{C}e_j + (\mathcal{C}e_k,f)_He_j \\
&= (\lambda_j+\lambda_k)(e_k,f)_He_j = (\lambda_j+\lambda_k)(e_j \otimes e_k)f.
\end{align*}
Thus, the family $(e_j \otimes e_k)_{(j,k)\in\mathbb N^2}$ consists of eigenfunctions of $\C_i$, $i=1,2$, where the eigenvalues are given by $\Lambda_{j,k} = \lambda_j\lambda_k$ when $i=1$ and $\Lambda_{j,k} = \lambda_j + \lambda_k$ when $i=2$. In particular $\C_1$ and $\C_2$ are diagonalizable. 
Now, for any $\epsilon > 0$, there are only finitely many $\lambda_j$, $j \geq 1$, which are larger than $\epsilon$ in absolute value. Thus, for any $\epsilon > 0$, the set $\{(j,k) \in \N^2 : |\lambda_j\lambda_k| > \epsilon\}$ is finite. Hence, according to Lemma 3.3.5 in Pedersen~\cite{Pedersen}, $\C_1$ is compact. On the other hand, if $|\lambda_1| > \epsilon > 0$, then the set $\{(j,k) \in \N^2 : |\lambda_j + \lambda_k| > \epsilon\}$ has infinitely many elements, since we may select $N \geq 1$ such that $|\lambda_1 + \lambda_k| > \epsilon$ holds for all $k \geq N$. Thus, according to Lemma 3.3.5 in Pedersen~\cite{Pedersen}, $\C_2$ is not compact.
\end{proof}

\begin{example}
Let $H=L^2([0,1])$ and define 
$$
(\mathcal{C}f)(t) = \int_0^1 (s \wedge t)f(s)ds,
$$
for $f \in H$. Then, since the kernel function $c(s,t) = s \wedge t$ is symmetric, the operator $\mathcal C$ is trace class (i.e. its series of eigenvalues is absolutely convergent so in particular $\mathcal{C}$ is compact), symmetric and non-negative-definite (see Theorem A.8 in Peszat and Zabczyk \cite{PeZa}). The kernel function $c(t,s)$ is moreover the covariance function of a univariate Brownian motion $c(t,s) = \E[B(t)B(s)]$. Now given this construction it follows by the Karhunen-Lo\`eve expansion that
$$
B(t) = \sum_j \sqrt{\lambda_j} \xi_j e_j(t),
$$
where $(\xi_j)_{j\in\mathbb N}$ are i.i.d. (orthonormal) standard normal random variables. In this case 
$$
e_j(t) = \sqrt{2}\sin\left(\frac{(2j + 1)\pi t}{2}\right), \ \ \lambda_j = \left(\frac{2}{(2j+1)\pi}\right)^2,
$$
for $j\geq 0$. Moreover, it follows from the Karhunen-Lo\`eve expansion that 
$$
(B,e_j)_H = \sqrt{\lambda_j} \xi_j.
$$
\end{example}

Suppose that $(E_j)_{j\in\mathbb N}$ is an ONB of $\mathcal{H}$, $\C \in L(\mathcal{H})$ and for each $n \geq 1$, $\mathcal{U}_n = \mathop{\rm span}\{E_j : j \in J_n\}$, where $J_n \subset \N$ and $(J_n)_{j\in\mathbb N}$ is a nested sequence of finite subsets of $\N$ such that $J_n \subset J_{n+1}$ for all $n \geq 1$. Define $\C^n$ by
\begin{equation}\label{def:CnOp}
    \C^n = \Pi_n\C\Pi_n,
\end{equation} 
where $\Pi_n \in L(\mathcal{H})$ is the orthogonal projection onto $\mathcal{U}_n$.
Note that since $L(\mathcal{H})$ is a Banach algebra and $\Pi_n$ is a contraction operator, it follows that
\begin{equation}\label{ineq:Cn}
    \|\C^n\|_\text{op} \leq \|\Pi_n\|_\text{op}^2 \|\C\|_\text{op} \leq \|\C\|_\text{op}, 
\end{equation}
so clearly, $\C^n \in L(\mathcal{H})$ and $\C^n\mathcal{T} \in \mathcal{U}_n$ for each $\mathcal{T} \in \mathcal{H}$.

\begin{prop}\label{prop:compact_C}
Suppose that $\C$ is a normal and compact operator. Then there exists an ONB $(E_j)_{j\in\mathbb N}$ of $\mathcal{H}$ consisting of eigenfunctions of $\C$ and $(\Lambda_{j})_{j\in\mathbb N}$ is the corresponding sequence of eigenvalues that vanishes at infinity. Assume that $\mathcal{U}_n = \mathop{\rm span}\{E_j : j \in J_n\}$, where $J_n \subset \N$ and $(J_n)_{n\in\mathbb N}$ is a nested sequence of finite subsets of $\N$ such that $J_n \subset J_{n+1}$ for all $n \geq 1$. Then  
$$
\|(\C - \C^n)\mathcal{T}\|_\mathcal{H}^2 = \sum_{j \in J_n^c } \Lambda_j^2 \langle \mathcal{T},E_j\rangle_\mathcal{H}^2,
$$
where $\C^n$ is defined by \eqref{def:CnOp} and $\mathcal{T} \in \mathcal{H}$. It furthermore holds that 
$$
\|\C - \C^n\|_\text{op}^2 \leq 2\sup_{m \in J_n^c} \Lambda_{m}^2 \to 0
$$ 
as $n \to \infty$.
\end{prop}
\begin{proof}
The existence of the ONB with eigenvalues that vanish at infinity follows by Theorem 3.3.8 in Pedersen~\cite{Pedersen}. Take $\mathcal T \in \mathcal H$, then we may write 
$$
\mathcal{T} = \sum_{j} \langle \mathcal{T}, E_j\rangle_\mathcal{H} E_j. 
$$
Since, $\C E_j = \Lambda_{j} E_j$ and
$$
\Lambda_j\langle \mathcal{T}, E_j\rangle_\mathcal{H} = \langle \mathcal{T}, \Lambda_j E_j\rangle_\mathcal{H} = \langle \mathcal{T}, \C E_j\rangle_\mathcal{H} = \langle \C^*\mathcal{T}, E_j\rangle_\mathcal{H}
$$
it follows by Parseval's identity that
\begin{align*}
    \|(\C - \C^n)\mathcal{T}\|_\mathcal{H}^2 &= \|\sum_{j \in J_n^c} \langle \mathcal{T}, E_j \rangle_\mathcal{H}  \Lambda_{j}E_j\|_\mathcal{H}^2 \\
    &= \|\sum_{j \in J_n^c} \langle \C^*\mathcal{T}, E_j \rangle_\mathcal{H}  E_j \|_\mathcal{H}^2 \\
    &= \sum_{j \in J_n^c}\langle \C^*\mathcal{T}, E_j \rangle_\mathcal{H}^2 = \sum_{j \in J_n^c } \Lambda_j^2 \langle \mathcal{T},E_j\rangle_\mathcal{H}^2. 
    \end{align*}
Since the eigenvalues vanish at infinity, $\Lambda_{j} \to 0$ as $j \to \infty$, it follows that 
\begin{align*} 
    \|\C - \C^n\|_\text{op}^2 &= \sup_{\|\mathcal{T}\|_{\mathcal{H}}=1} \sum_{j \in J_n^c} \Lambda_{j}^2 \langle \mathcal{T},E_j\rangle_\mathcal{H}^2 \\
    &\leq \sup_{\|\mathcal{T}\|_{\mathcal{H}}=1} \sum_{j \in J_n^c}  \langle \mathcal{T},E_j\rangle_\mathcal{H}^2 \sup_{m \in J_n^c} \Lambda_{m}^2 \\
    &= \sup_{\|\mathcal{T}\|_{\mathcal{H}}=1} \|\mathcal{T} - \mathcal{T}^n\|_\mathcal{H}^2 \sup_{m \in J_n^c} \Lambda_{m}^2 \leq 2\sup_{m \in J_n^c} \Lambda_{m}^2 \to 0,
\end{align*}
as $n \to \infty$.
\end{proof}
\begin{corollary}
Suppose that $\C$ is a normal and compact operator and with eigenvectors $(e_j \otimes e_k)_{(j,k)\in\mathbb N^2}$ and eigenvalues $(\Lambda_{j,k})_{(j,k)\in\mathbb N^2}$, where $(e_j)_{j\in\mathbb N}$ is an ONB of $H$. Assume that $\mathcal{U}_n = \mathop{\rm span}\{e_j \otimes e_k : j+k \leq n\}$. Then  
$$
\|(\C - \C^n)\mathcal{T}\|_\mathcal{H}^2 = \sum_{j+k > n} \Lambda_{j,k}^2 (\mathcal{T}e_k,e_j)_H^2,
$$
where $\C^n$ is defined by \eqref{def:CnOp} and $\mathcal{T} \in \mathcal{H}$. It furthermore holds that 
$$
\|\C - \C^n\|_\text{op}^2 \leq 2\sup_{j+k > n} \Lambda_{j,k}^2 \to 0,
$$ as $n \to \infty$.
\end{corollary}
\begin{proof}
Using the identity
$$
\langle \mathcal{T}, e_j \otimes e_k \rangle_\mathcal{H} = \sum_{m=1}^\infty (\mathcal{T}e_k, (e_m,e_k)_He_j)_H = (\mathcal{T}e_k,e_j)_H,
$$
we establish the first claim, the expression for the norm. Since the eigenvalues vanish at infinity according to Theorem 3.3.8 in Pedersen~\cite{Pedersen}, the second claim on the convergence of the operator norm follows by mimicking the proof of Proposition~\ref{prop:compact_C}
\end{proof}
We now proceed to the more general case, where $\C$ is not necessarily normal and compact. The following Lemma is useful in that setting.
\begin{lemma}\label{lem:op_Lip}
Suppose $A,B \in L(U)$ are linear operators on a Banach space $U$.  Then for any $k \geq 1$, it holds that
\begin{align*}
\|A^k - B^k\|_\text{op} \leq k (\|A\|_\text{op} \vee \|B\|_\text{op})^{k-1}\|A-B\|_\text{op}.
\end{align*}
\end{lemma}
\begin{proof}
Assume $\Vert B\Vert_{\text{op}}=0$. This implies $B=0$ and we have by the algebra-property of $L(U)$ that $\Vert A^k\Vert_{\text{op}}\leq\Vert A\Vert_{\text{op}}^k$, and the assertion holds. Suppose therefore, without loss of generality, that $\|A\|_\text{op} \geq \|B\|_\text{op}>0$. Note that by induction it holds that
\begin{align*}
A^k - B^k = \sum_{m=0}^{k-1} A^{k-1-m} (A-B) B ^m.
\end{align*}
Indeed, this holds for $k=1$, and if $k\geq 2$ is arbitrary, then by using the induction hypothesis,
\begin{align*}
A^k - B^k &= A(A^{k-1} - B^{k-1}) + (A-B)B^{k-1} \\
&= A \sum_{m=0}^{k-2} A^{k-2-m}(A-B)B^m + (A-B)B^{k-1} \\
&= \sum_{m=0}^{k-1} A^{k-1-m} (A-B) B ^m.
\end{align*}
Since the space of linear operators, $L(U)$, is a Banach algebra, it follows that
\begin{align*}
\|A^k - B^k\|_\text{op} &\leq  \|A\|_\text{op}^{k-1} \sum_{m=0}^{k-1} \left( \frac{\|B\|_\text{op}}{\|A\|_\text{op}} \right)^m \|A-B\|_\text{op} \\
&\leq k \|A\|_\text{op}^{k-1} \|A-B\|_\text{op}.
\end{align*}
The proof is finished by reversing the roles of $A$ and $B$, i.e. by assuming that $\|B\|_\text{op} \geq \|A\|_\text{op}$.
\end{proof}

In the next Proposition we state a general approximation estimate on $\mathcal V-\mathcal V^n$ in terms of $\mathfrak{C}-\mathfrak{C}^n$.
\begin{prop}\label{prop:V_CApprox}
Suppose that $\C, \C^n \in L(\mathcal{H})$. 
Then, if $\mathcal{V}(t)$ and $\mathcal{V}^n(t)$ are the variance processes defined by \eqref{def:variance-process-diff} and \eqref{def:variance-process-diff-approx} respectively, where $\mathcal{V}_0^n = \mathcal{V}_0$ and $\mathcal{L}^n = \mathcal{L}$ for all $n \geq 1$ and $\mathcal{L}(t)$ is a square integrable compound Poisson process. Then it holds that  
\begin{align*}
        \E\left[ \sup_{0 \leq t \leq T} \|\mathcal{V}(t) - \mathcal{V}^n(t) \|_\mathcal{H}^2 \right] &\leq C(T) \|\C - \C^n\|_\text{op}^2,
\end{align*}
where 
$
C(T) = 2T^2\e^{2T(\|\C\|_\text{op} \vee \|\C^n\|_\text{op})} (\E[\|\mathcal{V}_0\|_\mathcal{H}^2] + \lambda T  \E\left[\|\mathcal{X}_1\|_\mathcal{H}^2 \right] + \lambda^2 T^2 \left(\E\left[\|\mathcal{X}_i\|_\mathcal{H} \right]\right)^2).
$
\end{prop}
\begin{proof}
By employing the series representation of the exponential function, it follows from Lemma \ref{lem:op_Lip} that, if $t \geq 0$, then
\begin{align*}
    \|\e^{\C t} - \e^{\C^n t} \|_{\text{op}} &\leq  \sum_{k=0}^\infty \frac{t^k}{k!} \|\C^k - (\C^n)^k \|_\text{op} \leq t\e^{t(\|\C\|_\text{op} \vee \|\C^n\|_\text{op})} \|\C - \C^n\|_\text{op}.
\end{align*}
Hence, according to equation \eqref{def:variance_process} and a repeated application of the triangle inequality, we obtain that 
\begin{align*}
\| \mathcal{V}(t) - \mathcal{V}^n(t) \|_\mathcal{H} &\leq \|(\e^{\C t} - \e^{\C^n t})\mathcal{V}_0 \|_\mathcal{H} + \|\int_0^t (\e^{\C (t-s)} - \e^{\C^n (t-s)})d\mathcal{L}(s) \|_\mathcal{H} \\
&\leq \|\e^{\C t} - \e^{\C^n t}\|_\text{op} \|\mathcal V_0\|_\mathcal{H} + \sum_{i=1}^{N(t)} \|\e^{\C (t-\tau_i)} - \e^{\C^n (t-\tau_i)}\|_\text{op} \|\mathcal{X}_i\|_\mathcal{H} \\
&\leq t\e^{t(\|\C\|_\text{op} \vee \|\C^n\|_\text{op})} \left(  \|\mathcal V_0\|_\mathcal{H} + \sum_{i=1}^{N(t)}  \|\mathcal{X}_i\|_\mathcal{H} \right)\|\C - \C^n\|_\text{op}.
\end{align*}
Note furthermore that by the properties of a real-valued compound Poisson processes with intensity $\lambda$ it follows by Lemma~\ref{lem:U_CPP} that
\begin{align*}
    \E&\left[ \sup_{0 \leq t \leq T} \left( \|\mathcal V_0\|_\mathcal{H} + \sum_{i=1}^{N(t)}  \|\mathcal{X}_i\|_\mathcal{H} \right)^2\right]  \\
    &\qquad\qquad\leq 2\left( \E[\|\mathcal V_0\|_\mathcal{H}^2] + \E\left[ \left( \sup_{0 \leq t \leq T}\sum_{i=1}^{N(t)} \|\mathcal{X}_i\|_\mathcal{H} \right)^2 \right] \right) \\
    &\qquad\qquad= 2\left( \E[\|\mathcal V_0\|_\mathcal{H}^2] + \lambda T  \E\left[\|\mathcal{X}_1\|_\mathcal{H}^2 \right] + \lambda^2 T^2 \left(\E\left[\|\mathcal{X}_i\|_\mathcal{H} \right]\right)^2 \right).
\end{align*}
It follows that 
\begin{align*}
     \E\left[ \sup_{0 \leq t \leq T} \| \mathcal{V}(t) - \mathcal{V}^n(t) \|_\mathcal{H}^2 \right] &\leq C(T) \|\C - \C^n\|_\text{op}^2.
\end{align*}
The proof is complete.
\end{proof}
Combining the above result with Proposition \ref{prop:compact_C}, we can control the error by the tail convergence of the eigenvalues when $\mathfrak{C}$ is normal and compact. This is made precise in the next Corollary:
\begin{corollary}
\label{cor:var-est-L-V0}
Suppose that $\C$ is a normal and compact operator and with eigenvectors $(E_j)_{j\in\mathbb N}$ and eigenvalues $(\Lambda_{j})_{j\in\mathbb N}$ and assume that $\mathcal{T} \in \mathcal{H}$ and $\mathcal{U}_n = \mathop{\rm span}\{E_j : j \in J_n\}$, where $J_n \subset \N$ and $(J_n)_{n\in\mathbb N}$ is a nested sequence of finite subsets of $\N$ such that $J_n \subset J_{n+1}$ for all $n \geq 1$. Suppose $\mathcal{V}(t)$ and $\mathcal{V}^n(t)$ are the variance processes defined by \eqref{def:variance-process-diff} and \eqref{def:variance-process-diff-approx} respectively, where $\C^n$ is defined by \eqref{def:CnOp}, $\mathcal{V}_0^n = \mathcal{V}_0$ and $\mathcal{L}^n = \mathcal{L}$ for all $n \geq 1$ and $\mathcal{L}(t)$ is a square integrable compound Poisson process. Then it holds that
\begin{align*}
        \E\left[ \sup_{0 \leq t \leq T} \|\mathcal{V}(t) - \mathcal{V}^n(t) \|_\mathcal{H}^2 \right] &\leq C(T)\sup_{m \in J_n^c} \Lambda_{m}^2,
\end{align*}
where 
$
C(T) = 4T^2\e^{2T\|\C\|_\text{op}} (\E[\|\mathcal{V}_0\|_\mathcal{H}^2] + \lambda T  \E\left[\|\mathcal{X}_1\|_\mathcal{H}^2 \right] + \lambda^2 T^2 \left(\E\left[\|\mathcal{X}_i\|_\mathcal{H} \right]\right)^2)
$
and the right-hand sides of the above inequalities converge to zero as $n \to \infty$.
\end{corollary}
\begin{proof}
Follows from Propositions \ref{prop:compact_C} and \ref{prop:V_CApprox}, and the inequality \eqref{ineq:Cn}.
\end{proof}

Recalling the option pricing example discussed at the end of Section \ref{sec:Forwards}, we can now apply for example Proposition \ref{prop:V_CApprox} and Corollary \ref{cor:var-est-L-V0} to assess the pricing error $\vert P-P^n\vert$ in terms of the approximation error induced by $\mathfrak C^n$, or Proposition \ref{prop:Var_conv1} for the option price robustness towards errors in $\mathcal L$ and $\mathcal V_0$.

\section{The square root process}\label{sec:SqRootPr}
In this Section we study the robustness of the volatility process, being the square-root of the variance process. 

If an operator $A$ is self-adjoint and non-negative definite, then there exists a unique non-negative definite operator, $\sqrt{A}$, such that $\sqrt{A}\sqrt{A} = A$. Sufficient conditions that guarantee that $\mathcal{V}(t)$ in \eqref{def:variance-process-diff} is non-negative definite, as given by Benth, R\"udiger and S\"uss \cite{BeRuSu}, are that 
\begin{enumerate}
    \item[a)] $(\C \mathcal{T})^* = \C \mathcal{T}^*$,
    \item[b)] $\C \mathcal{H}_+ \subset \mathcal{H}_+$, where $\mathcal{H}_+$ denotes the convex cone of non-negative definite operators on $\mathcal{H}$,
    \item[c)] $\mathcal{L}(t)$ is a self-adjoint and non-negative definite square-integrable L\'evy process with values in $\mathcal{H}$, and
    \item[d)] $\mathcal{V}_0$ is self-adjoint and non-negative definite.
\end{enumerate}
In the following result we employ the operator norm to investigate the error induced on the square root of the variance process.
\begin{prop}\label{prop:RootOpNorm}
Suppose that $\mathcal{V}(t)$ and $\mathcal{V}^n(t)$ are non-negative definite. If $\mathcal{V}(t)$ and $\mathcal{V}^n(t)$ are the variance processes defined by \eqref{def:variance-process-diff} and \eqref{def:variance-process-diff-approx} respectively, then for every $T>0$, it holds that
\begin{equation*}
     \E\left[\sup_{0 \leq t \leq  T} \|\sqrt{\mathcal{V}(t)} - \sqrt{\mathcal{V}^n(t)}\|_{\mathrm{op}}^2\right] \leq
     \E\left[ \sup_{0 \leq t \leq T} \|\mathcal{V}(t) - \mathcal{V}^n(t)\|_\mathrm{op} \right].
\end{equation*}
\end{prop}
\begin{proof}
According to Lemma 2.5.1. in Bogachev \cite{Bogachev}
\begin{align*}
    \E\left[\sup_{0 \leq t \leq  T} \|\sqrt{\mathcal{V}(t)} - \sqrt{\mathcal{V}^n(t)}\|_{\text{op}}^2\right] 
    &\leq \E\left[ \sup_{0 \leq t \leq T} \|\mathcal{V}(t) - \mathcal{V}^n(t)\|_\text{op} \right], \\
\end{align*}
which concludes the proof.
\end{proof}

We remark that he inequality $\|\cdot\|_\text{op} \leq \|\cdot\|_\mathcal{H}$ can be employed to obtain the Hilbert-Schmidt norm on the right-hand side in the above result Proposition \ref{prop:RootOpNorm}. In the following result we obtain the following analogue of Proposition \ref{prop:RootOpNorm}, where the Hilbert-Schmidt norm is on the left-hand side and the trace class norm appears on the right-hand side. A compact operator is a trace class operator if it has an absolutely convergent series of eigenvalues and its norm is given by (see Section 2.5 in Bogachev \cite{Bogachev}) 
\begin{equation}\label{def:TraceNorm}
    \|\mathcal T\|_1 := \|(\mathcal T^* \mathcal T)^{1/4}\|_\mathcal{H}^2.
\end{equation}
\begin{prop}\label{prop:RootHSNorm}
Suppose that $\mathcal{V}(t)$ and $\mathcal{V}^n(t)$ are non-negative definite and trace class, where  $\mathcal{V}(t)$ and $\mathcal{V}^n(t)$ are the variance processes defined by \eqref{def:variance-process-diff} and \eqref{def:variance-process-diff-approx} respectively. Then for every $T>0$, it holds that
$$
 \E\left[ \sup_{0 \leq t \leq T} \|\sqrt{\mathcal{V}(t)} - \sqrt{\mathcal{V}^n(t)}\|_{\mathcal{H}}^2\right] \leq \E\left[  \sup_{0 \leq t \leq T}  \|\mathcal{V}(t) - \mathcal{V}^n(t)\|_1 \right]. 
$$
\end{prop}
\begin{proof}

According to Corollary 2 in Ando~\cite{Ando88}, which holds for non-negative and compact operators on $\mathcal{H}$ (see also Birman, Koplienko and Solomyak  \cite{BiKoSo75} for the initial result and Subsection 8.4 in Birman and Solomyak \cite{BiSo03} for an overview), it holds that if $\mathcal{A}, \mathcal{B}$ are non-negative definite and compact operators, then 
\begin{align*}
    \|\sqrt{\mathcal{A}} - \sqrt{\mathcal{B}}\|_{\mathcal{H}} &\leq \| |\mathcal{A} - \mathcal{B}|^{1/2} \|_{\mathcal{H}},
\end{align*}
where $|\mathcal{T}| = (\mathcal{T}^*\mathcal{T})^{1/2}$ denotes the modulus of an operator $\mathcal{T} \in \mathcal{H}$. Note that according to Proposition 3.4.8 in Pedersen all Hilbert-Schmidt operators are compact. Hence, according to the definition of the trace norm \eqref{def:TraceNorm} it follows that 
\begin{align*}
    \E\left[ \sup_{0 \leq t \leq T} \|\sqrt{\mathcal{V}(t)} - \sqrt{\mathcal{V}^n(t)}\|_{\mathcal{H}}^2\right] &\leq \E\left[  \sup_{0 \leq t \leq T}  \|\mathcal{V}(t) - \mathcal{V}^n(t)\|_1 \right].
\end{align*}
The proof is completed.

\end{proof}

We end this subsection with corollaries which detail how the the above result with the trace class norm on the right hand side can be employed when either the compound Poisson process or the generator, $\C$, are approximated.

\begin{corollary}
Suppose that $\mathcal{V}(t)$ and $\mathcal{V}^n(t)$ are non-negative definite and trace class, where  $\mathcal{V}(t)$ and $\mathcal{V}^n(t)$ are the variance processes defined by \eqref{def:variance-process-diff} and \eqref{def:variance-process-diff-approx} respectively, with $\mathcal{V}_0^n = \mathcal{V}_0$ and $\C^n = \C \in L(\mathcal{B}_1)$ for all $n \geq 1$, where $\mathcal{B}_1$ denotes the Banach space of trace class operators with norm \eqref{def:TraceNorm}, and the compound Poisson processes $\mathcal{L}(t)$ and $\mathcal{L}^n(t)$ are driven by the same Poisson process, $N(t)$, i.e. they jump simultaneously. Then for every $T>0$, it holds that
$$
 \E\left[\sup_{0 \leq t \leq  T} \|\sqrt{\mathcal{V}(t)} - \sqrt{\mathcal{V}^n(t)}\|_{\mathcal{H}}^2\right] \leq C(T)\E\left[\|\mathcal{X}_i - \mathcal{X}_i^n\|_1\right], 
$$
where $C(T) = k\e^{\|\C\|_{\text{op}}T} \lambda T$ and $k > 0$ is a constant.
\end{corollary}
\begin{proof}
 According to Lemma \ref{lem:Var_CPP} it holds that 
\begin{align*}
    \E\left[  \sup_{0 \leq t \leq T}  \|\mathcal{V}(t) - \mathcal{V}^n(t)\|_1 \right] 
    &\leq \E\left[  \sup_{0 \leq t \leq T}  \e^{\|\C\|_{\text{op}}t} \sum_{i=1}^{N(t)} \|\mathcal{X}_i - \mathcal{X}_i^n\|_{1} \right] \\
    &\leq \e^{\|\C\|_{\text{op}}T} \E\left[ \sum_{i=1}^{N(T)} \|\mathcal{X}_i - \mathcal{X}_i^n\|_1 \right] \\
    &= \e^{\|\C\|_{\text{op}}T}\lambda T \E\left[\|\mathcal{X}_i - \mathcal{X}_i^n\|_1\right].
\end{align*}
The result follows by Proposition \ref{prop:RootHSNorm}. 
\end{proof}

The following result reduces the trace class norm of a tensor product of Hilbert space elements to a product of their norms.

\begin{lemma}\label{lem:TraceNorm}
If $f,g \in H$, and $\mathcal T = f \otimes g$, so that $\mathcal{T}h = (g,h)_Hf$, for $h \in H$, then it holds that $\|\mathcal T\|_1 = |f|_H|g|_H$.
\end{lemma}
\begin{proof}
It is easy to show that $\mathcal{T}^* = g \otimes f$, and $\mathcal{T}^*\mathcal{T} = |f|_H^2 g^{\otimes 2}$. A simple calculation then shows that $(\mathcal{T}^*\mathcal{T})^{1/4} = |f|_H^{1/2}|g|_H^{-3/2} g^{\otimes 2}$. The result follows from an application of Parseval's identity:
$$
\|\mathcal{T}\|_1 = |f|_H|g|_H^{-3}\|g^{\otimes 2}\|_\mathcal{H}^2 = |f|_H|g|_H^{-1}\sum_{k=1}^\infty |(g,e_k)_H|^2 =  |f|_H|g|_H.
$$
\end{proof}

\begin{example}
Suppose that $\mathcal{X}_i = (Y_i)^{\otimes 2}$, where $(Y_i)_{i\in\mathbb N}$ is a sequence of i.i.d. $H$-valued random variables, and $\mathcal{X}_i^n = (Y_i^n)^{\otimes 2}$, where $Y_i^n$ is defined by \eqref{def:Y_approx}, for all $n,i \geq 1$. Then, according to Lemma \ref{lem:TraceNorm} it holds that 
\begin{align*}
    \|\mathcal{X}_i - \mathcal{X}_i^n\|_1 &\leq \|(Y_i-Y_i^n) \otimes Y_i\|_1 + \|Y_i^n \otimes (Y_i-Y_i^n)\|_1 \\
    &= (|Y_i|_H + |Y_i^n|_H)|Y_i-Y_i^n|_H \leq 2|Y_i|_H|Y_i-Y_i^n|_H.
\end{align*}
Hence, by appealing to the Cauchy-Schwarz inequality we obtain that 
$$
\E\left[\|\mathcal{X}_i - \mathcal{X}_i^n\|_1\right] \leq 2\left(\E[|Y_i|_H^2] \E[|Y_i-Y_i^n|_H^2] \right)^{1/2}.
$$
\end{example}

\begin{corollary}
Suppose that $\mathcal{V}(t)$ and $\mathcal{V}^n(t)$ are non-negative definite and trace class, where  $\mathcal{V}(t)$ and $\mathcal{V}^n(t)$ are the variance processes defined by \eqref{def:variance-process-diff} and \eqref{def:variance-process-diff-approx} respectively, with $\mathcal{V}_0^n = \mathcal{V}_0$, $\mathcal{L} = \mathcal{L}^n$ and $\C^n, \C \in L(\mathcal{B}_1)$ for all $n \geq 1$, where $\mathcal{B}_1$ denotes the Banach space of trace class operators with norm \eqref{def:TraceNorm}. Then for every $T>0$, it holds that
$$
 \E\left[\sup_{0 \leq t \leq  T} \|\sqrt{\mathcal{V}(t)} - \sqrt{\mathcal{V}^n(t)}\|_{\mathcal{H}}^2\right] \leq C(T)\|\C - \C^n\|_\mathrm{op}, 
$$
where $C(T) = kT \e^{T(\|\C\|_\text{op} \vee \|\C^n\|_\text{op})}\left(\E[\|\mathcal V_0\|_1] + \lambda T \E[\|\mathcal{X}_i\|_1 ]\right)$ and $k > 0$ is a constant.
\end{corollary}
\begin{proof}
Using the same reasoning as in the proof of Proposition \ref{prop:V_CApprox}, with the Hilbert-Schmidt norm replaced by the trace class norm it follows that 
\begin{align*}
\| \mathcal{V}(t) - \mathcal{V}^n(t) \|_1 
&\leq t\e^{t(\|\C\|_\text{op} \vee \|\C^n\|_\text{op})} \left(  \|\mathcal V_0\|_1 + \sum_{i=1}^{N(t)}  \|\mathcal{X}_i\|_1 \right)\|\C - \C^n\|_\text{op}.
\end{align*}
Hence, 
\begin{align*}
    \E\left[  \sup_{0 \leq t \leq T}  \|\mathcal{V}(t) - \mathcal{V}^n(t)\|_1 \right] 
    &\leq T \e^{T(\|\C\|_\text{op} \vee \|\C^n\|_\text{op})}\left(\E[\|\mathcal V_0\|_1] + \lambda T \E[\|\mathcal{X}_i\|_1 ]\right) \|\C - \C^n\|_\text{op}.
\end{align*}
The result follows by Proposition \ref{prop:RootHSNorm}.
\end{proof}

Cuchiero and Svaluto-Ferro \cite{CuSF} analyse options on the realised volatility (so-called VIX-options) in an infinite dimensional framework. One could introduce options on the realised volatility of forward prices, which would have a payoff $p(\mathcal D\mathcal V^{1/2}(\tau))$ at some exercise time $\tau$. Here, $\mathcal D\in\mathcal H^*$ maps the volatility operator $\mathcal V^{1/2}$ into the real line by an integral and evaluation operator, say. If the payoff function $p$ is Lipschitz continuous, we can apply the results of this section to assess the robustness of such volatility options with respect to the parameters $\mathfrak C$ and $\mathcal L$, both in operator and in Hilbert-Schmidt norms.     




\end{document}